\documentclass[11pt,a4paper]{amsart}
\let\amsmarkboth\markboth    
\usepackage{amsmath,amsfonts,amssymb,amsthm,amscd}
\usepackage[english]{babel}
\usepackage[mathscr]{eucal}
\usepackage{graphicx}
\usepackage{appendix}
\usepackage{color}

\makeatletter
\let\markboth\amsmarkboth
\bbl@redefine#1#2{%
   \def\bbl@arg{#1}%
   \ifx\bbl@arg\@empty
     \toks@{}%
   \else
     \toks@{\noexpand\foreignlanguage{%
              \languagename}{%
              \noexpand\bbl@restore@actives#1}}%
   \fi
   \def\bbl@arg{#2}%
   \ifx\bbl@arg\@empty
     \toks8{}%
   \else
     \toks8{\noexpand\foreignlanguage{%
              \languagename}{%
              \noexpand\bbl@restore@actives#2}}%
   \fi
   \edef\bbl@tempa{\the\toks@}%
   \edef\bbl@tempb{\the\toks8}%
   \protected@edef\bbl@tempa{%
     \noexpand\org@markboth{\bbl@tempa}{\bbl@tempb}}%
   \bbl@tempa
}
\DeclareRobustCommand*\ams@disablelinebreak{\def\\{ \ignorespaces}}
\def\maketitle{\par
   \@topnum\z@ %
   \@setcopyright
   \thispagestyle{firstpage}%
   \uppercasenonmath\shorttitle
   \ifx\@empty\shortauthors \let\shortauthors\shorttitle
   \else \andify\shortauthors
   \fi
   \@maketitle@hook
   \begingroup
   \@maketitle
   \toks@\@xp{\shortauthors}\@temptokena\@xp{\shorttitle}%
   \protected@edef\@tempa{%
     \@nx\markboth{\ams@disablelinebreak
       \@nx\MakeUppercase{\the\toks@}}{\the\@temptokena}}%
   \@tempa
   \endgroup
   \c@footnote\z@
   \@cleartopmattertags
}

\makeatother

\numberwithin{equation}{section}

\newcommand{\N}{{\mathbb N}}
\newcommand{\R}{{\mathbb R}}

\newcommand{\RR}{{\mathbb R}^2}

\newtheorem{theorem}{Theorem}[section]
\newtheorem{proposition}[theorem]{Proposition}
\newtheorem{lemma}[theorem]{Lemma}
\newtheorem{corollary}[theorem]{Corollary}

\theoremstyle{definition}
\newtheorem{definition}[theorem]{Definition}

\theoremstyle{remark}
\newtheorem{remark}[theorem]{Remark}

\newcommand{\dt}{\partial_t}

\newcommand{\eps}{\varepsilon}
\newcommand{\fe}{f_\eps}
\newcommand{\loc}{\text{loc}}

\begin{document}


\title[Gyrokinetic limit]{The gyrokinetic limit for the  Vlasov-Poisson system with a point charge}

\author[Evelyne Miot]{Evelyne Miot}
\address[E. Miot]{CNRS - Institut Fourier, UMR 5582\\ Universit\'e Grenoble-Alpes, France}  \email{evelyne.miot@univ-grenoble-alpes.fr}


\date{\today}

\begin{abstract}
We consider the asymptotics of large external magnetic field for a 2D Vlasov-Poisson system governing the evolution of a bounded density interacting with a point charge. 
We show that the solution converges to a measure-valued solution of the Euler equation with a defect measure.

\end{abstract}

\maketitle

\section{Introduction and main results}

\subsection{The gyrokinetic limit for the Vlasov-Poisson system with a point charge}

In this paper, we consider the asymptotical behavior of the solutions of a Vlasov-Poisson type system as $\eps$ tends to zero:
\begin{equation}
\label{syst:VP-bis}
\begin{cases}
\displaystyle \dt \fe+\frac{v}{\eps}\cdot \nabla_x \fe+\left(\frac{v^\perp}{\eps^2}+\frac{E_\eps}{\eps}+\frac{\gamma}{\eps} 
\frac{x-\xi_\eps}{|x-\xi_\eps|^2}\right)\cdot \nabla_v \fe=0,\quad (x,v)\in \R^2\times \R^2\\
\displaystyle E_\eps=\frac{x}{|x|^2}\ast \rho_\eps,\quad \text{where } \displaystyle \rho_\eps(t,x)=\int_{\R^2}\fe(t,x,v)\,dv\\
\displaystyle \dot{\xi}_\eps(t)=\frac{\eta_\eps(t)}{\eps},\\
\displaystyle \dot{\eta}_\eps(t)=\gamma\left(\frac{\eta_\eps^\perp(t)}{\eps^2}+\frac{E_\eps(t,\xi_\eps(t))}{\eps}\right),
\end{cases}
\end{equation}
with the initial conditions
\begin{equation}
\label{syst:ini}
\fe(0,x,v)=f_{\eps}^0(x,v),\quad (\xi_\eps,\eta_\eps)(0)=(\xi_\eps^0,\eta_\eps^0).
\end{equation}
Here, the real number $\gamma>0$ does not depend on $\eps$.
For each $\eps>0$, this system describes the interaction of a two-dimensional distribution of light 
particles (a plasma) and a heavy point charge $\gamma$, which are submitted
to a large and constant external magnetic field, orthogonal to the plane. More precisely, the distribution of particles is represented by the positive and bounded function $\fe=\fe(t,x,v)$, the point charge is located
at $\xi_\eps(t)$, with velocity $\eta_\eps(t)$. The particles are submitted to the self-consistent electric field $E_\eps$ on the one hand, and to the magnetic field
represented by the terms $v^\perp/\eps^2$ or $\eta_\eps^\perp/\eps^2$ on the other hand (here, $(x_1,x_2)^\perp=(-x_2,x_1)$).

For fixed $\eps>0$, the Cauchy theory for weak solutions of the classical Vlasov-Poisson system, namely \eqref{syst:VP-bis} without charge nor magnetic field, has been settled in several works \cite{Arsenev, ukai, LP, loeper}. 
Then, the Vlasov-Poisson system without magnetic field but with a point charge  was introduced by Caprino and Marchioro \cite{caprino-marchioro} (with $\eps=1$). For initial data satisfying
\begin{equation}\begin{split}
 \label{hyp:ini-0-bis}&f_\eps^0\in L^1\cap L^\infty(\RR\times \RR),\: f_\eps^0\geq 0,\:f_\eps^0 \text{  is compactly supported, } \\ 
 &\text{supp}(f_\eps^0)\subset \{(x,v)\in \RR \times \RR\:|\: |x-\xi_\eps^0|\geq \delta_\eps\}\quad \text{for some $\delta_\eps>0$},\end{split}\end{equation} 
global existence and uniqueness of a solution $(f_\eps,\xi_\eps)$ 
 with $f_\eps\in L^\infty(\R_+,L^1\cap L^\infty( \RR\times \RR))$ compactly supported was established in \cite{caprino-marchioro}. 
 We also refer to \cite{italiens-miot} for a related existence result in the case of attractive 
interaction between the plasma and the charge (namely if $\gamma<0$). This result can be easily extended to  \eqref{syst:VP-bis} with magnetic field, for each fixed $\eps>0$.

 The purpose of this paper is to investigate the asymptotics of \eqref{syst:VP-bis} for large external magnetic field, which corresponds to the limit $\eps$ tends to zero. We will show that 
 under suitable bounds on the initial data, the sequence $(\rho_\eps,\xi_\eps)$ is relatively compact for some suitable topology on measures and we will show in 
Theorem \ref{thm:main} that any
  accumulation point  $(\rho,\xi)$ satisfies the Euler equation \eqref{NLE}, with a defect measure. Furthermore, when the defect measure vanishes and under more regularity assumptions on $\rho$,  \eqref{NLE} yields a coupled system consisting in 
  a PDE for the evolution of $\rho$ and an ODE for the evolution of $\xi$:\begin{equation}
\label{syst:VW}
\begin{cases}
\displaystyle \dt \rho+\left(E^\perp+ \gamma \frac{(x-\xi)^\perp}{|x-\xi|^2}\right)\cdot \nabla \rho=0\\
\displaystyle \dot{\xi}(t)= E^\perp(t,\xi(t)),\quad E= \frac{ x}{|x|^2}\ast \rho.
\end{cases}
\end{equation}

\medskip

Before stating these theorems in subsection \ref{subsec:results}, we summarize the state of the art for the case without charge. The system \eqref{syst:VW} reduces then to the 2D incompressible Euler equation in vorticity formulation for the function $\rho$:
\begin{equation}
\label{eq:Euler}
\displaystyle \dt \rho+E^\perp\cdot \nabla \rho=0,\quad E= \frac{ x}{|x|^2}\ast \rho.
\end{equation}
In the periodic setting without charge,  Golse and Saint-Raymond \cite{golse-sr}, then Saint-Raymond \cite{SR-02} and also Brenier \cite{brenier} established 
the convergence of \eqref{syst:VP-bis} to the incompressible Euler equation under suitable assumptions on the initial data (see later). 
The same kind of result was recently obtained in \cite{miot-16}  by different techniques. Moreover, several asymptotical regimes 
for linear or non linear Vlasov-like equations, leading to various nonlinear equations,
were investigated in the articles \cite{frenod-sonnendrucker-98, frenod-sonnendrucker-99, frenod-sonnendrucker-01, golse-sr-2, SR-01, han-kwan, hauray-nouri}, and more recently 
in \cite{bostan-finot-hauray,barre-chiron-masmoudi}. Recently, the numerical  issues  
were studied by Filbet and Rodrigues \cite{filbet-rodrigues}, wo constructed an asymptotic-preserving
scheme for the Vlasov-Poisson system in the limit of large external magnetic field.
\medskip

We now turn to the system \eqref{syst:VW} also called vortex-wave system.
It was introduced by Marchioro and Pulvirenti \cite{marchioro-pulvirenti}, who established global existence and uniqueness 
of the solution such that $\rho\in L^\infty(\R_+,L^1\cap L^\infty(\RR))$ and $\xi\in C^1(\R_+)$ never intersects the support of $\rho$. It was later further 
analyzed in, e.g., \cite{lacave-miot, bresiliens-miot}. We will discuss below the possibility of giving a sense to \eqref{syst:VW}, 
or to \eqref{eq:Euler}, when $\rho$ is a measure. Our definition \ref{def:poupaud} below, borrowed from previous works, allows 
to handle \emph{vortex sheets}, namely measure-valued densities $\rho(t)$ belonging to $H^{-1}$.


 \subsection{Some notations}
 
 Throughout this paper,

\textbullet \, For $\Omega=\RR,$ $\Omega=\RR \times\RR$ or $\Omega=\mathbb{S}^1\times \RR$, $\mathcal{M}(\Omega)$ denotes the space of bounded real Radon measures and
 $\mathcal{M}_+(\Omega)$ the space of bounded, positive Radon measures on $\Omega$, $C_0(\Omega)$ the space of continous functions vanishing at infinity on $\Omega$. 
 We say that $\rho \in C_w(\R_+, \mathcal{M}_+(\Omega))$ if $\rho(t)\in \mathcal{M}_+(\Omega)$ for all $t\in \R_+$ and if moreover, $t\mapsto 
\int_{\Omega} \Phi(x)\,d\rho(t,x)$ is continous, for all $\Phi\in C_0(\Omega)$. The sequence $(\rho_n)_{n\in \N}$ is said to converge to $\rho$ in $C_w(\R_+,\mathcal{M}_+(\Omega)$ if
for all $T>0$ and for all $\Phi \in C_0(\Omega)$ we have
$\sup_{t\in [0,T]}\int_{\Omega}\Phi(x)(d\rho_n(t,x)-d\rho(t,x))\to 0$ as $n\to +\infty$. The sequence $(\rho_n)$ is said to converge to 
$\rho$ in $L^\infty(\R_+,\mathcal{M}_+(\Omega))$ weak - $\ast$ if for all $\Phi\in L^1(\R_+,C_0(\Omega))$ we have $\int_{\R_+}\int_{\Omega}\Phi(t,x)(d\rho_n(t,x)-d\rho(t,x))\to 0$
 as $n\to +\infty$.

\textbullet \,  For $A$, $B\in \mathcal{M}_{2,2}(\R)$ we set $A:B=\sum_{i,j}A_{i,j}B_{i,j}$ and for $x=(x_1,x_2)\in \RR$ we set $x\otimes x=x^tx=
\begin{pmatrix}x_1^2 & x_1x_2\\ x_1x_2 & x_2^2\end{pmatrix}. $ 

\textbullet \,  Except in the last section, $C$ denotes a constant changing possibly from a line to another, depending only on the uniform bounds on the initial data.

\subsection{Statements of the results}\label{subsec:results}
As already mentioned, the limits of the solutions of \eqref{syst:VP-bis} arising as $\eps\to 0$ are measure-valued.
In order to take into account such singular objects, we need to reexpress the nonlinear term $E^\perp\cdot \nabla \rho=\nabla\cdot(E^\perp \rho)$ in the sense of distributions. 
The formulation below, and some of its variants,
was  introduced by Schochet \cite{Schochet}, Delort \cite{Delort} or Poupaud \cite{poupaud} in the setting of weak solutions of the 2D Euler equation.
\begin{definition}[\cite{poupaud}, Def. 4.9]
 \label{def:poupaud}
Let $\rho,\mu \in \mathcal{M}_+(\R^2)$. For all $\Phi\in C_c^\infty(\R^2)$, we set
$$\mathcal{H}_\Phi[\rho,\mu]=\frac{1}{2}\iint_{\R^2\times \R^2} H_\Phi(x,y)d\rho(x)\,d\mu(y),$$
where
$$H_\Phi(x,y)=
\frac{(x-y)^\perp}{|x-y|^2}\cdot \left( \nabla \Phi(x)-\nabla \Phi(y)\right)\quad \text{if }x\neq y, \quad H_\Phi(x,x)=0.$$
\end{definition}

\begin{remark}

 \label{rem:bound} The map $(x,y)\mapsto H_\Phi(x,y)$ is defined and continuous off the diagonal $\Delta=\{(x,x)\:|\:x\in \RR\}$. It is also bounded on 
 $\RR\times \RR$ by the mean-value theorem,  hence the formulation above makes sense for $\rho$ and $\mu$ as in Definition \ref{def:poupaud}.

\end{remark}
The motivation of this definition is based on the following proposition, which is obtained by symmetrization of the variables $x$ and $y$. 
\begin{proposition}[\cite{poupaud,Delort, Schochet}] \label{prop:symm}In Definition \ref{def:poupaud} assume moreover that the measure $\rho$ belongs to $L^p(\RR)$ for some $p>2$. Then we have, recalling $E=\frac{x}{|x|^2}$,
$$\langle \nabla\cdot (E^\perp \rho),\Phi\rangle_{\mathcal{D}'(\RR),\mathcal{D}(\RR)}=- \mathcal{H}_{\Phi}[\rho,\rho].$$
 
\end{proposition}

\medskip

We clarify now our assumptions on the initial data. To $f\in L^1$, $\rho=\int f\,dv$, $\xi$ and $\eta\in \RR$ we associate the energy
\begin{equation*}\begin{split}
&\mathcal{H}(f,\xi,\eta)
=\frac{1}{2}\iint_{\R^2\times \R^2}|v|^2f(x,v)\,dx\,dv+\frac{1}{2}|\eta|^2\\
&-\frac{1}{2}\iint_{\R^2\times \RR}
\ln|x-y|\rho(x)\rho(y)\,dx\,dy
- \gamma \int_{\R^2}\ln|x-\xi|\rho(x)\,dx,
\end{split}
\end{equation*}
and the momentum
\begin{equation*}
\mathcal{I}(f,\xi,\eta)=\int_{\RR}\left( |x+\eps v^\perp|^2-\eps^2|v|^2\right)f(x,v)\,dx\,dv+\gamma |\xi+\frac{\eps}{\gamma}\eta^\perp|^2-\frac{\eps^2}{\gamma^2}|\eta|^2.
\end{equation*}As we shall later see, the energy and the momentum are preserved by the solutions of \eqref{syst:VP-bis} that are considered in this paper.
\medskip

Here we  restrict our attention to $\gamma>0$ and to initial data satisfying \eqref{hyp:ini-0-bis} for each $\eps>0$.  Moreover 
we assume the following behavior of the norms as $\eps \to 0$:
\begin{equation}\label{hyp:ini-1-bis}\begin{split}
&\sup_{0<\eps<1} \left( \|f_\eps^0\|_{L^1}+\int_{\RR}|x|^2\rho_\eps^0(x)\,dx+|\xi_\eps^0| \right)<+\infty,\\
&\sup_{0<\eps<1} \mathcal{H}(f_\eps^0,\xi_\eps^0,\eta_\eps^0)<+\infty,
\end{split}
\end{equation}
and 
\begin{equation}
\label{hyp:ini-2-bis}
 \eps^2 \|f_\eps^0\|_{L^\infty} \to 0,\quad \text{as }\eps\to 0.
\end{equation} Finally, we add the condition\footnote{More generally, the condition is
$\sup_{0<\eps<1}\|f_\eps^0\|_{L^1}<|\sigma|$, where $\sigma$ is such that $E=\sigma x/|x|^2\ast \rho$.}
:
\begin{equation}\label{hyp:ini-small} \sup_{0<\eps<1}\|f_\eps^0\|_{L^1}<1.\end{equation}

Our main result can now be stated as follows.
\begin{theorem}
 \label{thm:main}
Let $(\fe^0,\xi_\eps^0,\eta_\eps^0)$ satisfy \eqref{hyp:ini-0-bis}, \eqref{hyp:ini-1-bis}, \eqref{hyp:ini-2-bis} and \eqref{hyp:ini-small}. 
Let $(\fe,\xi_\eps)$ denote the corresponding global weak solution of \eqref{syst:VP-bis}. There exists a subsequence $\eps_n\to 0$ as $n\to +\infty$ such that

\textbullet \,  $(\rho_{\eps_n})$ converges to $\rho$ in $C_w(\R_+,\mathcal{M}_+(\RR))$ and
 $(\xi_{\eps_n})$ converges to $\xi$ in $C^{1/2}([0,T],\RR)$ for all $T>0$;

\textbullet \,    $\rho\in  L^\infty(\R_+,H^{-1}(\RR))$; 

 \textbullet \,  There exists a defect measure  $\nu\in [L^\infty(\R_+,\mathcal{M}(\RR)]^4$ such that $(\rho,\xi)$ satisfies: for all $\Phi\in C_c^\infty(\R_+\times \RR)$,
\begin{equation}\tag{E}\label{NLE}\begin{split}
\frac{d}{dt}\int_{\RR} \Phi(t,x) d(\rho(t) +\gamma \delta_{\xi(t)})(x)&=\int_{\RR} \partial_t \Phi(t,x) d(\rho(t)+\gamma\delta_{\xi(t)})(x)\\+\mathcal{H}_{\Phi(t)}[\rho
+ \gamma \delta_{\xi},\rho&
+\gamma \delta_{\xi}]
+\int_{\RR} D\nabla^\perp \Phi(t,x):d\nu(t,x)
\end{split}
\end{equation}
in the sense of distributions on $\R_+$.

\end{theorem}

The next theorem specifies the structure of the defect measure:

\begin{theorem}
 \label{thm:main-structure}Under the same assumptions as in Theorem \ref{thm:main}, 

\textbullet \,  There exists $\nu_0=\nu_0(t,x,\theta) \in L^\infty\left(\R_+,\mathcal{M}_+(\RR\times \mathbb{S}^1)\right)$
and there exists $(\alpha,\beta) \in L^\infty(\R_+, \R)^2$ such that 
$$\nu=\int_{\mathbb{S}^1}\theta\otimes \theta \:d\nu_0(\theta)
+\begin{pmatrix}-\beta \delta_\xi &\alpha \delta_\xi \\ \alpha \delta_\xi& 
\beta  \delta_{\xi}
\end{pmatrix}.$$ In particular, $\nu$ is symmetric.

\textbullet \,  The sequence $(f_{\eps_n})$ converges to $f=f(t,x,|v|)$ 
in $L^\infty(\R_+, \mathcal{M}_+(\RR \times \RR))$ weak - 
$\ast$ and  $\rho=\int f\,dv$. Moreover, for all  $\Phi$ continuous on $\mathbb{S}^1$, the sequence
$$\int_{\RR}\left(f_{\eps_n}(t,x,v)-f(t,x,|v|)\right)\Phi\left(\frac{v}{|v|}\right)|v|^2\,dv$$ converges to$$\int_{\mathbb{S}^1}\Phi(\theta)\: d\nu_0(\theta)$$ in the 
sense of distributions on $\R_+\times \RR$.

\end{theorem}

For the asymptotics without charge, Theorems \ref{thm:main} and \ref{thm:main-structure} were obtained by Golse and Saint-Raymond \cite[Theorem A]{golse-sr}, 
in which case $\nu$ reduces to $\nu_0$. The authors also derived some conditions ensuring that the defect measure is
rotation invariant, so that the terms $\int \theta_1\theta_2d\nu_0(\theta)$ and $\int (\theta_1^2-\theta_2^2)d\nu_0(\theta)$ eventually vanish. It would be interesting 
to study analog criteria for this so-called phenomenom of concentration-cancellation in the present case.

It was later proved by Saint-Raymond \cite{SR-01} that the defect measure vanishes, so that any accumulation point is a vortex-sheet solution of the Euler equation 
\eqref{eq:Euler}. The global existence of such solutions had been previously obtained by Delort \cite{Delort}.

\medskip

The equation \eqref{NLE} is nothing but the Euler equation \eqref{eq:Euler} for the total measure-valued 
vorticity $\omega=\rho+\gamma \delta_{\xi}$, according to the definition given by Poupaud \cite{poupaud}. We stress that such solutions however do not enter the framework of vortex-sheet solutions since Dirac masses do not belong to $H^{-1}$.

\medskip

Our last result shows that if there is no defect measure, assuming additional regularity on $\rho$ enables to decouple  the equation \eqref{NLE} 
to obtain the vortex-wave system
\begin{theorem}\label{thm:main-bis} Let $(\rho,\xi)$ be an accumulation point given by Theorem \ref{thm:main} and such that $\nu$ vanishes. 
If moreover $\rho\in L^\infty_\textrm{loc}(\R_+,L^p(\RR))$ for some $p>2$ and $\xi\in C^{1}(\R_+,\RR)$ then $(\rho, \xi)$ satisfies the system
\begin{equation*}
\begin{cases}
\displaystyle \dt \rho+\left(E^\perp+ \gamma \frac{(x-\xi)^\perp}{|x-\xi|^2}\right)\cdot \nabla \rho=0\\
\displaystyle \dot{\xi}(t)= E^\perp(t,\xi(t)),
\end{cases}
\end{equation*}where $E= \frac{ x}{|x|^2}\ast \rho$.
\end{theorem}
\begin{remark}\label{rem:bounded}
 It is classical (see e.g. \cite{livrejaune}) that if $\rho\in  L^\infty_\textrm{loc}(\R_+,L^1\cap L^p(\RR))$ for $p>2$ then $E= \frac{ x}{|x|^2}\ast \rho$ belongs to 
 $L^\infty_\textrm{loc}(\R_+,C^{0,1-\frac{2}{p}})\cap L^\infty_\textrm{loc}(\R_+,L^\infty(\RR))$.
\end{remark}

\medskip

The plan of the paper is the following. In the next section we establish Theorems \ref{thm:main} and \ref{thm:main-structure}, decomposing the proof into several steps. 
We first look for a priori
estimates with respect to $\eps$. Then, the main argument is the weak formulation satisfied by $(f_\eps, \xi_\eps)$, derived in Proposition 
 \ref{prop:weak-formulation-2}, in which we eventually pass to the limit by using the a priori estimates. Finally we prove Theorem \ref{thm:main-bis}.

\section{Proofs of Theorem \ref{thm:main} and Theorem \ref{thm:main-structure}}

Throughout this section, we consider a sequence of initial data $(f_\eps^0,\xi_\eps^0,\eta_\eps^0)$ satisfying the assumptions \eqref{hyp:ini-0-bis} \eqref{hyp:ini-1-bis}, 
\eqref{hyp:ini-2-bis} and \eqref{hyp:ini-small}. Let $(f_\eps,\xi_\eps)$ denote any corresponding sequence of global weak solutions to \eqref{syst:VP-bis}. 

\medskip

We will sometimes denote by
\begin{equation}
\label{def:L}
L_\eps(t,x)= \gamma \frac{x-\xi_\eps(t)}{|x-\xi_\eps(t)|^2}\end{equation} the singular  electic field generated by the point charge.

\subsection{Lagrangian trajectories}
The same arguments as in \cite{caprino-marchioro} imply that the unique solution $f_\eps$ to the system \eqref{syst:VP-bis} is constant along 
the Lagrangian trajectories associated to the field $E_\eps+L_\eps$. More precisely, we have the representation 
\begin{equation}
\label{push-forward}
\fe(t)=(X_\eps(t), V_\eps(t))_\#f_\eps^0,\end{equation}
where for all $(x,v)\in \R^2\times \R^2$, the map 
$t\mapsto (X_\eps(t,x,v),V_\eps(t,x,v))$ belongs to $W^{1,\infty}(\R_+,\RR\times \RR)$ and is the unique absolutely continuous solution to the ODE 
\begin{equation}\label{syst:ODE}
\begin{cases}
\displaystyle \dot{X}_\eps(t,x,v)=\frac{1}{\eps} V_\eps(t,x,v)\\
\displaystyle \dot{V}_\eps(t,x,v)=\frac{1}{\eps^2}\left(V_\eps^\perp(t,x,v)+\eps (E_\eps+L_\eps)(t,X_\eps(t,x,v))\right),
\end{cases}
\end{equation}
with $(X_\eps,V_\eps)(0,x,v)=(x,v).$
Moreover, the repulsive interaction (recall $\gamma>0$) between the plasma and the charge ensures that if $x\neq \xi_\eps^0$ then $X_\eps(t,x,v)\neq \xi_\eps(t)$ for all $t>0$ 
(see the proof of \cite[Corollary 2.4]{caprino-marchioro}), so that $t\mapsto L_\eps(t,X_\eps(t,x,v))$ - and therefore also 
the  flow map $(X_\eps,V_\eps)$ - is globally defined in time.

Note that since $f_\eps$ has compact velocity support, then $\rho_\eps$ belongs to $L^\infty_\loc(\R_+,L^\infty(\RR))$ for all $\eps>0$. It follows in particular that $E_\eps$ belongs to $L^\infty_\loc(\R_+,L^\infty(\RR))$ as well (note that its norm may blow up as $\eps$ tends to zero) and that it is almost-Lipschitz: $|E_\eps(t,x)-E_\eps(t,y)|\leq 
C_{\eps}|x-y|(1+|\ln |x-y||)$ (see e.g. \cite[(46)]{LP}). Thus it turns out that 
for all $(x,v)\in \RR \times \RR$ the map $t\mapsto (X_\eps(t,x,v),V_\eps(t,x,v))$ belongs to $W^{1,\infty}(\R_+,\RR\times \RR)$. 
Finally, noticing that $f_\eps$ is in $C_w(\R_+,L^p)$ for all $1<p< +\infty$ because of \eqref{push-forward}, it is not difficult to infer that 
$t\mapsto E_\eps(t,x)$ is continuous in time, uniformly with respect to $x$, so that finally, $t\mapsto (X_\eps(t,x,v),V_\eps(t,x,v))$ is the unique 
$C^1$ solution to the ODE \eqref{syst:ODE}.

\subsection{First a priori estimates}

As a starting point we gather some useful facts, most of them are standard and we only sketch the proofs.
\begin{proposition}\label{prop:uni} We have for all $t>0$:
\begin{equation*}\begin{split}
& \mathcal{H}(f_\eps(t),\xi_\eps(t),\eta_\eps(t))=
 \mathcal{H}(f_\eps^0,\xi_\eps^0,\eta_\eps^0);\quad
\|f_\eps(t)\|_{L^p}=\|f_\eps^0\|_{L^p},\quad \forall 1\leq p\leq +\infty;\\
& \mathcal{I}(f_\eps(t),\xi_\eps,\eta_\eps(t))=\mathcal{I}(f_\eps(0),\xi_\eps(0),\eta_\eps(0)).
\end{split}
\end{equation*}
\end{proposition}
\begin{proof} By straightforward computations we show that $\dot{\mathcal{H}}(f_\eps(t),\xi_\eps(t),\eta_\eps(t))=0$.  
The conservation of the norms is a consequence of \eqref{push-forward}. Finally, adapting the proof of  \cite[Proposition 2.3]{miot-16} 
to the present case with point charge, we show that $\dot{\mathcal{I}}(f_\eps(t),\xi_\eps(t),\eta_\eps(t))=0$.
\end{proof}

\begin{corollary}\label{coro:uni}
We have\begin{equation*}\begin{split}
&\sup_{t\in \R_+}\sup_{\eps>0}\left( \iint_{\RR\times \RR}|v|^2 f_\eps(t,x,v)\,dx\,dv+\int_{\RR}|x|^2\rho_\eps(t,x)\,dx\right)<+\infty,\\
&\sup_{t\in \R_+}\sup_{\eps>0}\left(|\xi_\eps(t)|+\eta_\eps(t)|\right)<+\infty,
\end{split}
\end{equation*}and 
\begin{equation*}
\sup_{t\in \R_+}\|\rho_\eps(t)\|_{L^2}\leq C \|f_\eps^0\|_{L^\infty}^{1/2}.
\end{equation*}
Finally,
\begin{equation*}\begin{split}
&\sup_{t\in \R_+}\sup_{\eps>0}\left| \iint_{\RR\times \RR}\ln|x-y|\rho_\eps(t,x)\rho_\eps(t,y)\,dx\,dy\right|<+\infty.
\end{split}
\end{equation*}

\end{corollary}
\begin{proof} The first estimate is established in \cite{caprino-marchioro} for finite interval of times. For a global in time estimate, we adapt easily the case without charge, which was handled in \cite[Proposition 2.4]{miot-16}, in the following way. We omit below the dependence upon $t$ when not misleading. Setting $$K_\eps=\iint_{\RR\times \RR}|v|^2 f_\eps(x,v)\,dx\,dv+|\eta_\eps|^2$$ we have by Proposition \ref{prop:uni}, by Cauchy-Schwarz inequality and by the uniform bound \eqref{hyp:ini-1-bis},
\begin{equation*}
\begin{split}
K_\eps&\leq2\mathcal{H}(f_\eps,\eta_\eps,\xi_\eps)+\iint_{\RR\times \RR}\ln_+|x-y|\rho_\eps(x)\rho_\eps(y)\,dx\,dy
+2\gamma \int_{\RR}\ln_+|x-\xi_\eps|\rho_\eps(x)\,dx\\
&\leq 2\mathcal{H}(f_\eps,\eta_\eps,\xi_\eps)(0)+C\iint_{\RR\times \RR}\left(|x|+|y|\right)\rho_\eps(x)\rho_\eps(y)\,dx\,dy
+C |\xi_\eps|\|\rho_\eps\|_{L^1}.\end{split}\end{equation*}
We have used that $\ln_+r\leq r$ in the second inequality. Thus by \eqref{hyp:ini-1-bis},
\begin{equation}\begin{split}\label{ineq:K}
K_\eps &\leq C+C\left(|\xi_\eps|^2+\int_{\RR}|x|^2\rho_\eps(x)\,dx\right)^{1/2}.
\end{split}
\end{equation}  Applying again
Proposition \ref{prop:uni},  we have
\begin{equation*}
\begin{split}&\gamma |\xi_\eps|^2+\int_{\RR}|x|^2\rho_\eps(x)\,dx= \mathcal{I}(f_\eps,\xi_\eps,\eta_\eps)+2\eps\int_{\RR} x\cdot v^\perp f_\eps(x,v)\,dx\,dv +2\eps \xi_\eps\cdot \eta_\eps ^\perp\\
&\leq   \mathcal{I}(f_\eps^0,\xi_\eps^0,\eta_\eps^0)+\frac{1}{2}\left(|\xi_\eps|^2+\int_{\RR}|x|^2\rho_\eps(x)\,dx\right)
+ C\eps^2 K_\eps\\
&\leq C\left( |\xi_\eps^0|^2+\int_{\RR}|x|^2\rho_\eps^0(x)\,dx\right)+ C\eps^2 K_\eps^0+\frac{1}{2}\left(|\xi_\eps|^2+\int_{\RR}|x|^2\rho_\eps(x)\,dx\right)
+ C\eps^2 K_\eps.
\end{split}
\end{equation*}
By \eqref{ineq:K} and by \eqref{hyp:ini-1-bis} we get
\begin{equation*}
\begin{split}\gamma |\xi_\eps|^2+\int_{\RR}|x|^2\rho_\eps(x)\,dx
&\leq C+\frac{1}{2}\left(|\xi_\eps|^2+\int_{\RR}|x|^2\rho_\eps(x)\,dx\right)\\
&+ C\eps^2 \left(|\xi_\eps|^2+\int_{\RR}|x|^2\rho_\eps(x)\,dx\right)^{1/2},
\end{split}
\end{equation*}
therefore
\begin{equation*}
\begin{split}|\xi_\eps|^2+\int_{\RR}|x|^2\rho_\eps(x)\,dx\leq C,
\end{split}
\end{equation*}so that also $K_\eps\leq C$.

The second estimate is classical, see e.g. \cite[Lemma 3.1]{golse-sr} or \cite[Lemma 2.4]{SR-02}: one has the interpolation inequality 
\begin{equation*}
\|\rho_\eps\|_{L^2}\leq C \|\fe\|_{L^\infty}^{1/2} \left(\iint_{\RR\times \RR}|v|^2 f_\eps(t,x,v)\,dx\,dv\right)^{1/2}
\end{equation*}
and the estimate $K_\eps\leq C$ yields the result. 

Finally, we obtain the last estimate by noticing that the left-hand-side can be estimated in terms of $\|\rho_\eps\|_{L^2}$, $\int |x|^2 \rho_\eps(x)\,dx$ and $\|\rho_\eps\|_{L^1}$. The conclusion follows.
\end{proof}

As in \cite{miot-16}, we introduce a smooth, positive function $\overline{\rho}_\eps$, compactly supported in $B(0,1)$, such that $\int \overline{\rho}_\eps=\int \rho_\eps$ and $\sup_{0<\eps<1}\|\overline{\rho}_\eps\|_{L^\infty}<+\infty$. Setting  $\overline{E}_\eps=(x/|x|^2)\ast \overline{\rho}_{\eps}$ it is well-known that $\sup_{0<\eps<1}\|\overline{E}_\eps\|_{L^\infty}<+\infty$ and that $E_\eps(t)-\overline{E}_\eps$ belongs to $L^2(\RR)$, 
see e.g. \cite[Proposition 3.3]{majda-bertozzi}. 
In addition, by combining the previous estimates in Proposition \ref{prop:uni} and Corollary \ref{coro:uni} we get 
\begin{equation}\label{ineq:e-barre}
\sup_{t\in \R_+}\sup_{0<\eps<1}
\|E_\eps(t)-\overline{E}_\eps\|_{L^2}<+\infty
\end{equation}
(see the proof of  \cite[Proposition 2.5]{miot-16}). We remark that $(\rho_\eps)$ is therefore uniformly bounded in $L^\infty(\R_+,H^{-1}(\RR))$ since $E_\eps(t)-\overline{E}_\eps=2\pi\nabla \Delta^{-1}(\rho_\eps(t)-\overline{\rho}_\eps).$

\medskip

We conclude this paragraph with a non concentration property that will be useful later. The following lemma was proved by Majda \cite{majda-93} (page 932). 
Other variants, using $L^2$ norm of the field, were established in \cite{Delort, Schochet}.
 \begin{proposition}
 \label{prop:delort}
  Let $\rho\in \mathcal{M}_+(\RR)$ such that $I=\int_{\RR}|x|^2 \rho(x)\,dx<+\infty.$ Assume that $${H}(\rho)=\iint_{\RR\times \RR}|\ln|x-y||\rho(x)\rho(y)\,dx\,dy<+\infty.$$
  Then there exists $C>0$ depending only on $\int \rho, I$ and ${H}$, such that for all $0<r<1/2$ we have
  \begin{equation*}
   \sup_{x_0\in \RR}\int_{B(x_0,r)}\rho(x)\,dx\leq C|\ln r|^{-1/2}.
  \end{equation*}
 
 \end{proposition}
  In particular, it follows directly from Corollary \ref{coro:uni}  that
  \begin{proposition}\label{prop:delort-2}
   There exists $C>0$ such that
   \begin{equation*}
  \sup_{t\in \R_+}\sup_{0<\eps<1}\sup_{x_0\in \RR}    \sup_{0<r<1/2} |\ln r|^{1/2} \int_{B(x_0,r)}\rho_\eps(t,x)\,dx<+\infty.
   \end{equation*}
 
  \end{proposition}

\subsection{Some estimates for the charge}

In this paragraph we focus on the dynamics of the charge by looking for estimates on the time integral  $\int E_\eps(\xi_\eps)\,dt$.
\begin{proposition}\label{prop:virial}
Let $(x,v)\in \R^2\times \R^2$. Then for all $t\in \R_+$,
\begin{equation*}\begin{split}
\frac{\gamma}{|X_\eps(t,x,v)-\xi_\eps(t)|}&\leq \eps^2\frac{d^2}{dt^2}|X_\eps(t,x,v)-\xi_\eps(t)|+\frac{|V_\eps(t,x,v)-\gamma \eta_\eps(t)|}{\eps}\\
&+|E_\eps(t,X_\eps(t,x,v)|+\gamma |E_\eps(t,\xi_\eps(t)|.\end{split}
\end{equation*}

\end{proposition}

\begin{proof}We compute, writing $(X_\eps,V_\eps)=(X_\eps(t,x,v),V_\eps(t,x,v))$ for simplicity,
\begin{equation*}
\begin{split}
\frac{d}{dt}|X_\eps-\xi_\eps|=\left(\frac{X_\eps-\xi_\eps}{|X_\eps-\xi_\eps|},\frac{V_\eps-\eta_\eps}{\eps}\right),
\end{split}
\end{equation*}
so
\begin{equation*}
\begin{split}
\frac{d^2}{dt^2}|X_\eps-\xi_\eps|
&=\frac{|V_\eps-\eta_\eps|^2}{\eps^2 |X_\eps-\xi_\eps|}+\frac{1}{\eps^3}
\left( \frac{X_\eps-\xi_\eps}{|X_\eps-\xi_\eps|},V_\eps^\perp-\gamma \eta_\eps^\perp\right)\\
&+\frac{1}{\eps^2}\left(\frac{X_\eps-\xi_\eps}{|X_\eps-\xi_\eps|},E_\eps(X_\eps)-\gamma E_\eps(\xi_\eps)\right)+\frac{\gamma}{\eps^2}\frac{1}{|X_\eps-\xi_\eps|}\\
&-\frac{1}{\eps}\frac{\left(X_\eps-\xi_\eps,V_\eps-\eta_\eps\right)}{|X_\eps-\xi_\eps|^2}\left(\frac{X_\eps-\xi_\eps}{|X_\eps-\xi_\eps|},\frac{V_\eps-\eta_\eps}{\eps}\right),
\end{split}
\end{equation*}
therefore
\begin{equation*}
\begin{split}
\frac{d^2}{dt^2}|X_\eps-\xi_\eps|&\geq 
\frac{|V_\eps-\eta_\eps|^2}{\eps^2 |X_\eps-\xi_\eps|}-\frac{1}{\eps^2}
\frac{\left(X_\eps-\xi_\eps,V_\eps-\eta_\eps\right)^2}{|X_\eps-\xi_\eps|^3}\\
&-\frac{1}{\eps^3}|V_\eps-\gamma \eta_\eps|-\frac{1}{\eps^2}|E_\eps(X_\eps)|-
\frac{\gamma}{\eps^2}|E_\eps(\xi_\eps)|+\frac{\gamma}{\eps^2}\frac{1}{|X_\eps-\xi_\eps|}
\end{split}
\end{equation*}
and the conclusion follows.
\end{proof}

\begin{corollary}\label{coro:virial}We have
\begin{equation*}
\begin{split}
|E_\eps(t,\xi_\eps(t))|&\leq C\eps^2 \frac{d^2}{dt^2}\iint_{\RR\times \RR}|x-\xi_\eps(t)|f_\eps(t,x,v)\,dx\,dv+\frac{C}{\eps}.
\end{split}
\end{equation*}
\end{corollary}

\begin{proof}
Integrating the inequality given by Proposition \ref{prop:virial} with respect to the measure $f_\eps^0(x,v)\,dx\,dv$ we get after changing variable
\begin{equation*}
\begin{split}
\gamma \int_{\RR}\frac{\rho_\eps(t,x)}{|x-\xi_\eps(t)|}\,dx&\leq
\eps^2\frac{d^2}{dt^2}\iint_{\RR\times \RR}
|x-\xi_\eps(t)|f_\eps(t,x,v)\,dx\,dv\\&+\frac{1}{\eps}\iint_{\RR\times \RR}
|v-\gamma \eta_\eps|f_\eps(t,x,v)\,dx\,dv\\
&+\int_{\RR}|E_\eps(t,x)|\rho_\eps(t,x)\,dx+\|f_\eps^0\|_{L^1}\gamma |E_\eps(t,\xi_\eps)|.
\end{split}
\end{equation*}
On the one hand, we have by \eqref{ineq:e-barre}
\begin{equation*}\begin{split}
\int_{\RR}|E_\eps(t,x)|\rho_\eps(t,x)\,dx
&\leq \|E_\eps(t)-\overline{E}_\eps\|_{L^2}\|\rho_\eps(t)\|_{L^2}
+\|\overline{E}_\eps\|_{L^\infty}\|f_\eps^0\|_{L^1}\\
&\leq C(\|f_\eps^0\|_{L^\infty}^{1/2}+\|f_\eps^0\|_{L^1})\leq \frac{C}{\eps},\end{split}
\end{equation*}where we have used the bounds \eqref{hyp:ini-1-bis} and  \eqref{hyp:ini-2-bis} in the last inequality.

On the other hand, Corollary \ref{coro:uni} yields by Cauchy-Schwarz inequality
\begin{equation*}
\begin{split}
&\iint_{\RR\times \RR}|v-\gamma \eta_\eps(t)|f_\eps(t,x,v)\,dx\,dv\\&\leq C \left(\iint_{\RR\times \RR}|v|^2f_\eps(t,x,v)\,dx\,dv+|\eta_\eps(t)|^2\|f_\eps^0\|_{L^1}\right)^{1/2}\|f_\eps(t)\|_{L^1}^{1/2}
\leq C.
\end{split}
\end{equation*}

Finally,
\begin{equation*}
\|f_\eps^0\|_{L^1}\gamma
|E_\eps(t,\xi_\eps)|\leq \sup_{0<\eps<1}\|f_\eps^0\|_{L^1}\gamma \int_{\RR}\frac{\rho_\eps(t,x)}{|x-\xi_\eps(t)|}\,dx,
\end{equation*}thus we obtain the desired estimate in view of the assumption \eqref{hyp:ini-small}.

\end{proof}

\begin{corollary}
 \label{coro:virial-3}We have 
\begin{equation*}
\begin{split}
\int_s^t |E_\eps(\tau,\xi_\eps(\tau))|\,d\tau&\leq C\eps +\frac{C}{\eps}(t-s).
\end{split}
\end{equation*}
\end{corollary}
\begin{proof}
We set 
$$I_\eps(t)=\iint_{\RR\times \RR}|x-\xi_\eps(t)|f_\eps(t,x,v)\,dx\,dv=
\iint_{\RR\times \RR}|X_\eps(t,x,v)-\xi_\eps(t)|f_\eps^0(x,v)\,dx\,dv,$$ so that by the system \eqref{syst:ODE}, using again the estimates of Corollary \ref{coro:uni} we get
\begin{equation*}\label{ineq:der-I}
\left|\frac{d}{dt}I_\eps(t)\right|\leq \frac{1}{\eps}\iint_{\RR\times \RR}|V_\eps(t,x,v)-\eta_\eps(t)|f_\eps^0(x,v)\,dx\,dv\leq \frac{C}{\eps}.
\end{equation*}
Then, integrating in time the inequality in Corollary \ref{coro:virial} yields
 \begin{equation*}
  \begin{split}
   \int_s^t |E_\eps(\tau,\xi_\eps(\tau))|&\leq C\eps^2\left(\left|\frac{d I_\eps}{dt}(t)\right|+\left|\frac{d I_\eps}{dt}(s)\right|\right)+\frac{C(t-s)}{\eps},
  \end{split}
\end{equation*}which implies the claim of the corollary.
\end{proof}

\subsection{Weak formulation}In order to study the asymptotical equation for \eqref{syst:VP-bis},
 we reexpress the system \eqref{syst:VP-bis}, using a weak formulation that was derived in 
\cite{golse-sr} as a starting point  in the case without charge.

\begin{proposition}
 \label{prop:weak-formulation-2}
 We have\begin{equation*}\begin{split}
                          \partial_t \rho_\eps&+\nabla_x\cdot\left((E_\eps^\perp+L_\eps^\perp)\rho_\eps \right)\\
                          &=\nabla_x\cdot \left( \left[\nabla_x\cdot \int_{\R^2}v\otimes v f_\eps \,dv\right]^\perp\right)+\eps \nabla_x\cdot \partial_t \int_{\RR}v^\perp f_\eps \,dv,
                          \end{split}
                          \end{equation*}where we recall the definition \eqref{def:L} for $L_\eps$.
\end{proposition}
\begin{proof}
 See the equations (3.8) and (3.9) in the proof of Lemma 3.2 in \cite{golse-sr}, substituting $E_\eps$ by $E_\eps+L_\eps$.
\end{proof}

In order to deal with the singular term $L_\eps^\perp \rho_\eps $ as $\eps$ tends to zero, for $\rho_\eps$ converging to a Radon measure,  we shall actually symmetrize the nonlinear 
term as 
in Definition \ref{def:poupaud} with
respect to the total measure $\rho_\eps+ \gamma \delta_{\xi_\eps}$. This can be done
 by taking into account the dynamics of the charge.
\begin{proposition}
 \label{prop:weak-formulation-3}
 Let $\Phi \in C_c^\infty(\R_+\times  \R^2)$. We have for all $t\geq 0$
 \begin{equation*}\begin{split}
 &\int_{\RR}\Phi(t,x)\rho_\eps(t,x)\,dx+ \gamma \Phi(t,\xi_\eps(t))-\int_{\RR}\Phi(0,x)\rho_\eps(0,x)\,dx-\gamma \Phi(0,\xi_\eps(0))\\
 &=\int_0^t \int_{\RR}\partial_t\Phi(s,x) \rho_\eps(s,x)\,ds\,dx+\gamma\int_0^t \partial_t \Phi(s,\xi_\eps(s))\,ds\\
 &+\int_0^t \mathcal{H}_{\Phi(s,\cdot)}[\rho_\eps(s,\cdot)+ \gamma\delta_{\xi_\eps(s)},\rho_\eps(s,\cdot)+ \gamma\delta_{\xi_\eps(s)}]\,ds\\
 &-\int_0^t \int_{\RR}\left(D\nabla^\perp\Phi(s,x):\int_{\RR}v\otimes v f_\eps(s,x,v)\,dv\right)\,dx\,ds\\
 &- \int_0^t \eta_\eps(s)\cdot (D\nabla^\perp \Phi(s,\xi_\eps(s))\: \eta_\eps(s))\,ds\\
&+ R_\eps(t),
                          \end{split}
                          \end{equation*}
 where
 \begin{equation*}
  |R_\eps(t)|\leq C\|\Phi\|_{W^{3,\infty}(\R_+\times \RR)}(1+t)\eps.
 \end{equation*}
We recall that $\mathcal{H}_{\Phi}[\cdot,\cdot]$ is defined in Definition \ref{def:poupaud}.

\end{proposition}

\begin{proof}
We apply Proposition \ref{prop:weak-formulation-2} with the test function $\Phi$. After symmetrizing the term $ E_\eps^\perp \rho_\eps$ as in Definition \ref{def:poupaud}, we obtain
 \begin{equation}\label{eq:w1}\begin{split}
 &\int_{\RR}\Phi(t,x)\rho_\eps(t,x)\,dx-\int_{\RR}\Phi(0,x)\rho_\eps(0,x)\,dx
 =\int_0^t \int_{\RR}\partial_t\Phi(s,x) \rho_\eps(s,x)\,ds\,dx\\
 &+\int_0^t \mathcal{H}_{\Phi(s,\cdot)}[\rho_\eps(s),\rho_\eps(s)]\,ds\\
 &+\int_0^t \int_{\RR} L_\eps^\perp(s,x)\cdot \nabla \Phi(s,x)\rho_\eps(s,x)\,dx\,ds\\
 &-\int_0^t \int_{\RR}\left(D\nabla^\perp\Phi(s,x):\int_{\RR}v\otimes v f_\eps(s,x,v)\,dv\right)\,dx\,ds+R_{\eps}^1,\end{split}\end{equation}where
 \begin{equation*}\begin{split}
  R_{\eps}^1 &=\eps \int_0^t \iint_{\RR}  f_\eps(s,x,v)v^\perp \cdot \partial_t \nabla \Phi(s,x)\,dv\,dx\,ds\\
&-\eps \iint_{\RR \times \RR}f_\eps(t,x,v)v^\perp \cdot \nabla \Phi(t,x)\,dx\,dv+\eps \iint_{\RR \times \RR}f_\eps^0(x,v)v^\perp \cdot \nabla \Phi(0,x)\,dx\,dv,
                          \end{split}
                          \end{equation*}
so that by Corollary \ref{coro:uni},
                          \begin{equation*}
                           |R_{\eps}^1|\leq C\eps\left(t \|\partial_t \nabla \Phi\|_{L^\infty}+\|\nabla \Phi\|_{L^\infty}\right).
                          \end{equation*}

                          \medskip

 Next, we insert the motion of the point charge. We introduce the combination 
$$h_\eps(t)=\xi_\eps(t)+\frac{\eps}{\gamma} \eta_\eps(t)^\perp,$$
so that by the mean-value theorem and by Corollary \ref{coro:uni} again, 
\begin{equation*}
  \begin{split}
\gamma   \Phi(t,\xi_\eps(t))- \gamma \Phi(0,\xi_\eps(0))&=\gamma \Phi(t,h_\eps(t))-\gamma \Phi(0,h_\eps(0))+R_{\eps}^2,\end{split}\end{equation*}
   with $$|R_{\eps}^2|\leq C\eps \|\nabla \Phi\|_{L^\infty}.$$ 
On the other hand, we observe that
$$\dot{h}_\eps(t)= E_\eps^\perp(t,\xi_\eps(t)),$$
therefore
   \begin{equation*}
   \begin{split}  
\gamma   \Phi(t,h_\eps(t))-\gamma \Phi(0,h_\eps(0))&=
\gamma   \int_0^t \partial_t \Phi(s,h_\eps(s))\,ds+ \gamma \int_0^t E_\eps^\perp(s,\xi_\eps(s))\cdot \nabla \Phi(s,h_\eps(s))\,ds\\
   &=\gamma \int_0^t \partial_t \Phi(s,\xi_\eps(s))\,ds+ \gamma \int_0^t E_\eps^\perp(s,\xi_\eps(s))\cdot \nabla \Phi(s,\xi_\eps(s))\,ds\\
   &+\gamma \int_0^t E_\eps^\perp(s,\xi_\eps(s))\cdot [\nabla \Phi(s,h_\eps(s))-\nabla \Phi(s,\xi_\eps(s))]\,ds\\
   &+R_{\eps}^3,
  \end{split}
 \end{equation*}
 where $$|R_{\eps}^3|\leq C \eps t( \|\partial_t \nabla \Phi\|_{L^\infty}+\|\partial_t  \Phi\|_{L^\infty}).$$
 
 \medskip
 
Moreover, we observe that 
\begin{equation*}\begin{split} 
 \int_{\RR} L_\eps^\perp(s,x)\cdot &\nabla \Phi(s,x)\rho_\eps(s,x)\,dx+\gamma  E_\eps^\perp(s,\xi_\eps(s))\cdot \nabla \Phi(s,\xi_\eps(s))\\
&= \int_{\RR} L_\eps^\perp(s,x)\cdot \big[\nabla \Phi(s,x)-\nabla \Phi(s,\xi_\eps(s))\big]
\rho_\eps(s,x)\,dx\\
&=2\mathcal{H}_{\Phi(s,\cdot)}[\rho_\eps(s), \gamma \delta_{\xi_\eps(s)}].\end{split}\end{equation*}

Noticing that 
\begin{equation*}
\begin{split}
\mathcal{H}_{\Phi}[\rho_\eps+\gamma \delta_{\xi_\eps},\rho_\eps+\gamma \delta_{\xi_\eps}]
&=\mathcal{H}_{\Phi}[\rho_\eps,\rho_\eps]+2 \mathcal{H}_{\Phi}[\rho_\eps,\gamma \delta_{\xi_\eps}]
\end{split}
\end{equation*}
and inserting this latter in \eqref{eq:w1} we obtain
                    \begin{equation}\label{eq:w2}\begin{split}
 &\int_{\RR}\Phi(t,x)\rho_\eps(t,x)\,dx-\int_{\RR}\Phi(0,x)\rho_\eps(0,x)\,dx+ \gamma \Phi(t,\xi_\eps(t))-\gamma \Phi(0,\xi_\eps(0))\\
& =\int_0^t \int_{\RR}\partial_t\Phi(s,x) \rho_\eps(s,x)\,ds\,dx+\gamma \int_0^t \partial_t \Phi(s,\xi_\eps(s))\,ds\\
 &+\int_0^t \mathcal{H}_{\Phi(s,\cdot)}[\rho_\eps(s)+\gamma \delta_{\xi_\eps(s)},\rho_\eps(s)+\gamma \delta_{\xi_\eps(s)}]\,ds\\
 &-\int_0^t \int_{\RR}\left(D\nabla^\perp\Phi(s,x):\int_{\RR}v\otimes v f_\eps(s,x,v)\,dv\right)\,dx\,ds\\
 &+\gamma\int_0^t E_\eps^\perp(s,\xi_\eps(s))\cdot [\nabla \Phi(s,h_\eps(s))-\nabla \Phi(s,\xi_\eps(s))]\,ds\\
 &+R_{\eps}^1+ R_{\eps}^2+ R_{\eps}^3.
 \end{split}\end{equation}
                    
                   \medskip

We next estimate the last (non remainder) term in \eqref{eq:w2}.
                We have
                    \begin{equation*}
                     \begin{split}
                \gamma     \int_0^t E_\eps^\perp(s,\xi_\eps(s))&\cdot [\nabla \Phi(s,h_\eps(s))-\nabla \Phi(s,\xi_\eps(s))]\,ds
                     \\&= \int_0^t E_\eps^\perp(s,\xi_\eps(s))\cdot [D\nabla \Phi(s,\xi_\eps(s))\: \eps \eta_\eps^\perp(s)]\,ds+R_{\eps}^4,
                     \end{split}\end{equation*}
with $$|R_{\eps}^4|\leq C\eps^2\sup_{t\in \R_+}|\eta_\eps(t)|^2\|D^3\Phi\|_{L^\infty}
\int_0^t |E_\eps(s,\xi_\eps(s))|\,ds\leq C\eps \|D^3\Phi\|_{L^\infty}(\eps^2+t),$$
where we used  Corollary \ref{coro:virial-3}.

\medskip

We now claim that
\begin{equation}\label{eq:claim}\begin{split}
\int_0^t  E_\eps^\perp(s,\xi_\eps(s))\cdot& [D\nabla \Phi(s,\xi_\eps(s))\: \eps \eta_\eps^\perp(s)]\,ds\\
&=-\int_0^t \eta_\eps(s)\cdot  [D\nabla^\perp \Phi(s,\xi_\eps(s))\:  \eta_\eps(s)]\,ds+R_{\eps}^5,\end{split}\end{equation}where
$$|R_{\eps}^5|\leq C\eps (1+t) \|\Phi\|_{W^{3,\infty}},$$
 which together with \eqref{eq:w2} and the estimates on the remainders will yield the conclusion of the proposition.

\medskip

\noindent \textbf{Proof of  \eqref{eq:claim}.}
Recalling  that $\eps E_\eps(\xi_\eps)^\perp=\eta_\eps+\frac{\eps^2}{\gamma}\dot{\eta}_\eps^\perp,$ we have
\begin{equation*}
 \begin{split}
& \int_0^t \eps E_\eps^\perp(s,\xi_\eps(s))\cdot [D\nabla \Phi(s,\xi_\eps(s))\:  \eta_\eps^\perp(s)]\,ds
 \\&=\int_0^t \eta_\eps(s)\cdot  [D\nabla \Phi(s,\xi_\eps(s))\:  \eta_\eps^\perp(s)]\,ds 
 +\frac{\eps^2}{\gamma} \int_0^t \dot{\eta}_\eps ^\perp(s) \cdot [D\nabla \Phi(s,\xi_\eps(s))\:  \eta_\eps^\perp(s)]\,ds\\
 &=I+J.\end{split}\end{equation*}

 On the one hand, for $F=\nabla \Phi(s,\cdot)$, a simple computation shows that
$$a\cdot (DF\: a^\perp)=-a\cdot (D(F^\perp)\: a),\quad \forall a\in \RR,$$
hence
$$I=-\int_0^t \eta_\eps(s)\cdot  [D\nabla^\perp \Phi(s,\xi_\eps(s))\:  \eta_\eps(s)]\,ds .$$

On the other hand, integrating by parts in $J$ we get 
 \begin{equation}\label{ineq:1}\begin{split}
J&=-\frac{\eps^2}{\gamma} \int_0^t {\eta}_\eps^\perp(s) \cdot \frac{d}{ds}[D\nabla \Phi(s,\xi_\eps(s))
 \:  \eta_\eps^\perp(s)]\,ds+R_{\eps}^6,\end{split}\end{equation}where \begin{equation*}
 R_{\eps}^6=\frac{\eps^2}{\gamma} {\eta}_\eps^\perp(t) \cdot [D\nabla \Phi(t,\xi_\eps(t))\:  \eta_\eps^\perp(t)]-
 \frac{\eps^2}{\gamma} {\eta}_\eps^\perp(0) \cdot [D\nabla \Phi(0,\xi_\eps(0))\:  \eta_\eps^\perp(0)]
\end{equation*}
so that
$$|R_\eps^6|\leq C\eps^2 \|D^2 \Phi\|_{L^\infty}.$$
Next, we compute 
\begin{equation*}
 \begin{split}
&  -\frac{\eps^2}{\gamma} \int_0^t {\eta}_\eps ^\perp(s) \cdot \frac{d}{ds}[D\nabla \Phi(s,\xi_\eps(s))
 \:  \eta_\eps^\perp(s)]\,ds\\
 &=-\frac{\eps^2}{\gamma} \int_0^t {\eta}_\eps ^\perp(s) \cdot [\frac{d}{ds}(D\nabla \Phi(s,\xi_\eps(s)))
 \:  \eta_\eps^\perp(s)]\,ds
 -\frac{\eps^2}{\gamma}\int_0^t {\eta}_\eps ^\perp(s) \cdot [D\nabla \Phi(s,\xi_\eps(s))
 \:  \dot{\eta}_\eps^\perp(s)]\,ds
 \end{split}\end{equation*}
 and using that $|\dot{\xi}_\eps|=|\eta_\eps|/\eps\leq C/\eps$ in the first term of the RHS we obtain
 \begin{equation*}
 \begin{split}
&  -\frac{\eps^2}{\gamma} \int_0^t {\eta}_\eps^\perp(s) \cdot \frac{d}{ds}[D\nabla \Phi(s,\xi_\eps(s))
 \:  \eta_\eps^\perp(s)]\,ds\\
 &=-\frac{\eps^2}{\gamma}\int_0^t {\eta}_\eps ^\perp(s) \cdot [D\nabla \Phi(s,\xi_\eps(s))
 \:  \dot{\eta}_\eps^\perp(s)]\,ds+R_{\eps}^7,
 \end{split}
\end{equation*}
                    with $$|R_{\eps}^7|\leq                     Ct\: \eps\: \|D^3 \Phi\|_{L^\infty}. $$

                    Coming back to \eqref{ineq:1}, we obtain therefore
                    \begin{equation*}\begin{split}
                    J=-\frac{\eps^2}{\gamma}
                     \int_0^t {\eta}_\eps ^\perp(s) \cdot [D\nabla \Phi(s,\xi_\eps(s))
 \:  \dot{\eta}_\eps^\perp(s)]\,ds+R_{\eps}^8
                     \end{split}
                    \end{equation*}
with $$|R_{\eps}^8|\leq C\eps(1+t) \| \Phi \|_{W^{3,\infty}}.$$
Now, we observe that for all $F\in C^1(\RR,\RR)$, we have
$$a\cdot( D F \: b)=b\cdot (DF\: a)+\text{curl}(F) a^\perp\cdot b,\quad \forall (a,b)\in \RR\times \RR,$$
where we have defined $\text{curl}(F)=\partial_2 F_1-\partial_1 F_2.$ Applying this to $F=\nabla \Phi(s,\cdot)$, so that $\text{curl}(F)=0$, to 
$a=\eta_\eps^\perp$ and $b=\dot{\eta}_\eps^\perp$, we get $$J=-J+R_{\eps}^8,$$ hence \eqref{eq:claim} follows.

\end{proof}

\subsection{Estimate on the trajectory of the charge}

\begin{corollary}\label{coro:holder}Let $T>0$. There exists $K_0>1$ and $\eps_0>0$, depending only on $T$, such that for  all $0<\eps<\eps_0$ and for all $0\leq s<t\leq T$,
\begin{equation}
|\xi_\eps(t)-\xi_\eps(s)|\leq K_0 \left((t-s)^{1/2}+\eps^{1/3}\right).
\end{equation}

\end{corollary}

\begin{proof}By Proposition \ref{prop:weak-formulation-3} and by Remark \ref{rem:bound}, we have for all $\Phi\in C_c^\infty(\RR)$ and for all $0\leq s<t\leq T$
\begin{equation*}
\begin{split}
&\left|\int_{\RR}\Phi(x)\rho_\eps(t,x)\,dx+ \Phi(\xi_\eps(t))-
\int_{\RR}\Phi(x)\rho_\eps(s,x)\,dx-\Phi(\xi_\eps(s))\right|\\
&\leq C\|\Phi\|_{W^{2,\infty}}(t-s)\sup_{\tau \in [0,T]}
\left(\|\rho_\eps(\tau)\|_{L^1}^2+\|\rho_\eps(\tau)\|_{L^1}+\iint_{\RR}|v|^2f_\eps(\tau,x,v)\,dx\,dv+|\eta_\eps(\tau)|^2\right)\\
&+|R_\eps(t)|+|R_\eps(s)|,
\end{split}
\end{equation*}
and by Corollary \ref{coro:uni} it follows that
\begin{equation}\label{ineq:holder}
\begin{split}
&\left|\int_{\RR}\Phi(x)\rho_\eps(t,x)\,dx+ \Phi(\xi_\eps(t))-
\int_{\RR}\Phi(x)\rho_\eps(s,x)\,dx- \Phi(\xi_\eps(s))\right|\\
&\leq C\|\Phi\|_{W^{2,\infty}}(t-s)+C(1+T)\:\eps\: \|\Phi\|_{W^{3,\infty}}.
\end{split}
\end{equation}

Let $K_1>1$ be a sufficiently large number to be determined later, depending only on $T$. Let $\eps_0=K_1^{-6}$.  We first claim that
\begin{equation}\label{claim:holder}\begin{split}&\forall 0<\eps<\eps_0,\quad \forall 0\leq s<t\leq T\quad \text{with }t-s\leq K_1^{-4},\\
&|\xi_\eps(t)-\xi_\eps(s)|\leq 2K_1\left((t-s)^{1/2}+\eps^{1/3}\right).\end{split}\end{equation}
                                                                                   Otherwise, there exist $0<\eps<\eps_0$ and
$0\leq s<t\leq T$ with $t-s\leq K_1^{-4}$ but
$|\xi_\eps(t)-\xi_\eps(s)|>2K_1\left((t-s)^{1/2}+\eps^{1/3}\right)$. We set 
$$\Phi(x)=\chi\left(\frac{x-\xi_\eps(s)}{K_1\left((t-s)^{1/2}+\eps^{1/3}\right)}\right),$$where  $\chi$ is a cut-off function 
such that $\chi=1$ on $B(0,1)$ and $\chi$ vanishes on $B(0,2)^c$. In particular, we have $\Phi(\xi_\eps(t))=0$ and $\Phi(\xi_\eps(s))=1$. 
Moreover, since $K_1(t-s)^{1/2}<1$ and $K_1\eps^{1/3}\leq K_1^{-1}<1$, we have
$$\|\Phi\|_{W^{2,\infty}}\leq CK_1^{-2}(t-s)^{-1},\quad \|\Phi\|_{W^{3,\infty}}\leq CK_1^{-3}\eps^{-1}.$$ In view of \eqref{ineq:holder} and using Proposition \ref{prop:delort-2}, 
we get
\begin{equation*}
\begin{split}
1&\leq 2\sup_{\tau\in [0,T]}\int_{B(\xi_\eps(\tau),2K_1\left((t-s)^{1/2}+\eps^{1/3}\right)}\rho_\eps(\tau,x)\,dx+C(1+ T)\:K_1^{-3}
\\
&\leq 2\sup_{\tau\in [0,T]}\int_{B(\xi_\eps(\tau),2K_1^{-3})}\rho_\eps(\tau,x)\,dx+C(1+ T)\:K_1^{-3}
\\
&\leq C\left(|\ln (2K_1^{-3})|^{-1/2}+(1+T)\:K_1^{-3}\right)\leq \frac{1}{2}
\end{split}
\end{equation*}if we choose $K_1$ sufficiently large (note that this choice may be done explicit). This yields a 
contradiction, and \eqref{claim:holder} follows.

\medskip

Now, we split $[0,T]$ as $[0,T]=\cup_{i=0}^{N-1} [t_i,t_{i+1}]$, with $|t_{i+1}-t_i|=K_1^{-4}$, for $i=1,\ldots N-1$, and $|t_1-t_0|\leq K_1^{-4}$. 
Let $0<\eps<\eps_0$. Let $0\leq s<t\leq T$ such that $|t-s|>K_1^{-4}$ and $i<j$ such that $t\in [t_i,t_{i+1})$ and $s\in [t_j,t_{j+1})$. We have by \eqref{claim:holder} 
\begin{equation*}
\begin{split}
|\xi_\eps(t)-\xi_\eps(s)|&\leq |\xi_\eps(t)-\xi_\eps(t_i)|+|\xi_\eps(t_i)-\xi_\eps(t_j)|+|\xi_\eps(t_j)-\xi_\eps(s)|\\
&\leq 2K_1\left(|t-t_i|^{1/2}+2\eps^{1/3}+|s-t_j|^{1/2}\right)\\
&+2K_1(N+1)\left(K_1^{-2}+\eps^{1/3}\right)\\
&\leq 2K_1^{-1}(N+3)+2K_1(N+3)\eps^{1/3}\\
&\leq 2K_1(N+3) \left(|t-s|^{1/2}+\eps^{1/3}\right).
\end{split}
\end{equation*}
Taking $K_0=2K_1(N+3)$, we are led to the result.
\end{proof}

\subsection{Time equicontinuity for the densities}

In this paragraph we prove the following
\begin{lemma}
 \label{lemma:equicontinuity}Let $T>0$. There exists $K_1>0$ and $\eps_0>0$, depending only on $T$, such that for all $0<\eps<\eps_0$ and for all $0\leq s<t\leq T$, 
 \begin{equation*}
  \|\rho_\eps(t)-\rho_\eps(s)\|_{W^{-3,1}(\RR)}\leq K_1\left((t-s)^{1/2}+\eps^{1/3}\right).
 \end{equation*}

\end{lemma}

\begin{proof} As for Corollary \ref{coro:holder}, the proof relies on Proposition \ref{prop:weak-formulation-3} and Remark \ref{rem:bound}. By \eqref{ineq:holder} we have 
for all $\Phi\in C_c^\infty(\RR)$ and for all $0\leq s<t\leq T$
\begin{equation*}
\begin{split}
&\left|\int_{\RR}\Phi(x)\rho_\eps(t,x)\,dx
-\int_{\RR}\Phi(x)\rho_\eps(s,x)\,dx\right|\\
&\leq \left| \Phi(\xi_\eps(t))- \Phi(\xi_\eps(s))\right|+ C\|\Phi\|_{W^{2,\infty}}(t-s)+C(1+T)\:\eps\:\|\Phi\|_{W^{3,\infty}}.
\end{split}
\end{equation*} In view of Corollary \ref{coro:holder}, this yields by the mean-value theorem
\begin{equation*}\begin{split}
 \Big|\int_{\RR}\Phi(x)\rho_\eps(t,x)\,dx&
-\int_{\RR}\Phi(x)\rho_\eps(s,x)\,dx\Big|\leq K_0\|\nabla \Phi\|_{L^{\infty}}
\left( (t-s)^{1/2}+\eps^{1/3}\right)\\
&+C\|\Phi\|_{W^{2,\infty}}(t-s)
+C(1+T)\:\eps\:\|\Phi\|_{W^{3,\infty}},\end{split}
\end{equation*}
from which the conclusion follows.

\end{proof}

\begin{remark}
 In \cite{golse-sr} (see also \cite{miot-16}), it is proved instead that the sequence of densities is uniformly bounded in $C^{1/2}(\R_+,W^{-2,1}(\RR))$. Here 
 we loose one derivative, due to 
 the contribution of the point charge appearing in the estimate for the remainder in the proof of Proposition \ref{prop:weak-formulation-3}.
\end{remark}

\subsection{Compactness}\label{subsec:compactness}
In this paragraph we use the previous estimates to show that
\begin{proposition}
\label{prop:compactness}
 There exists a subsequence such that $(\rho_{\eps_n})$ converges to some $\rho$ in $C_w(\R_+,\mathcal{M}_+(\RR))$ as $n\to +\infty$, and $\rho$ belongs to $L^\infty(\R_+,H^{-1}(\RR))$. The sequence $(\xi_{\eps_n})$ 
converges to some $\xi$ in $C^{1/2}([0,T],\RR)$ for all $T>0$.
\end{proposition}

To show Proposition \ref{prop:compactness} we shall use a straightforward adaptation of Ascoli's theorem:
\begin{lemma}\label{lemma:ascoli-bis}Let $T>0$. Let $(F,d)$ be a complete metric space.
 Let $(g_\eps)$ be a family of $C([0,T], F)$ such that 
 \begin{enumerate}
  \item For all $t\in [0,T]$, the family $(g_\eps(t))$ is relatively compact in $F$;
  \item There exists $C>0$ and a sequence $r_\eps \to 0$ as $\eps \to +0$ such that for all $t,s\in [0,T]$, for all $\eps>0$, 
  $d(g_\eps(t),g_\eps(s))\leq C|t-s|^{1/2}+r_\eps$.
 \end{enumerate}Then the family $(g_\eps)$ is relatively compact in $C([0,T],F)$.
\end{lemma}

Recalling that  $(\rho_\eps)$ is uniformly bounded in $L^\infty(\R_+, \mathcal{M}_+(\RR))$ and in view of Lemma \ref{lemma:equicontinuity}, 
we can apply this Lemma to $F=W^{-3,1}$ for any $T>0$.  Arguing as in the proof of Lemma 3.2 in \cite[Lemma 3.2]{Schochet} and using a diagonal argument, we can then show the existence of
$\eps_n\to 0$ as $n\to +\infty$ such that $(\rho_{\eps_n})$ converges to some $\rho$ in $C_w(\R_+,\mathcal{M}_+(\RR))$.  By \eqref{ineq:e-barre}, the family $(\rho_\eps)$ is bounded 
in $L^\infty(\R_+,H^{-1}(\RR))$ so we infer that $\rho$ belongs to $L^\infty(\R_+,H^{-1}(\RR))$.

\medskip

On the other hand, the sequence $(\xi_{\eps_n})$ is uniformly bounded in $L^\infty(\R_+,\RR)$ in view of
Corollary \ref{coro:uni}. Recalling that
Corollary \ref{coro:holder} holds, applying
Lemma \ref{lemma:ascoli-bis} and a diagonal argument, we obtain a subsequence, still denoted in the same way, such that $(\xi_{\eps_n})$ 
converges to some $\xi$ in $C^{1/2}([0,T],\RR)$ for all $T>0$.
Therefore the proposition is proved.

\subsection{Existence of a defect measure}

The following lemma in an extension of Lemma 3.3 in \cite{golse-sr}.

\begin{lemma}[\cite{golse-sr}, Lemma 3.3]\label{lemma:radial}
 Under the assumptions of Theorem \ref{thm:main}, the sequence $(f_{\eps_n})$ is relatively compact in $L^\infty(\R_+,\mathcal{M}(\RR\times \RR))$ weak - $*$. 
 Moreover, any accumulation point $f$ satisfies 
$$\nabla_v\cdot(v^\perp f)=0,\quad \text{in }\mathcal{D}'(\R_+^\ast \times \RR \times \RR).$$
\end{lemma}

\begin{proof} We follow the arguments of \cite{golse-sr}. 
We have
\begin{equation*}\begin{split}
                  v^\perp\cdot \nabla_v f_\eps=-\partial_t(\eps^2f_\eps)-\nabla_x\cdot(\eps v f_\eps)-\nabla_v\cdot (\eps (L_\eps +E_\eps) f_\eps).
                 \end{split}
\end{equation*}
By Corollary \ref{coro:uni}, the first two terms of the RHS converge to zero in the sense of distributions on $\R_+^\ast \times \RR \times \RR$. We next focus on the last  
term. Let 
$\Phi$ be a test function with support included in $[0,R]\times B_{\RR}(0,R)^2$ for some $R>0$. We have by \eqref{ineq:e-barre}
\begin{equation*}
 \begin{split}
&\eps  \left|\int_0^{+\infty}
  \iint_{\RR\times \RR} E_\eps(t,x)\cdot \nabla_v \Phi(t,x,v)f_\eps(t,x,v)\,dx\,dv\,dt\right|\\&
  \leq \eps \|\nabla_v\Phi\|_{L^\infty} R\left(\|\rho_\eps\|_{L^\infty(L^2)}\|E_\eps-\overline{E}_\eps\|_{L^\infty(L^2)}+\|\rho_\eps\|_{L^\infty(L^1)}\|\overline{E}_\eps\|_{L^\infty}\right)\\
  &\leq C\|\nabla_v\Phi\|_{L^\infty}\left( \eps \|f_\eps^0\|_{L^\infty}^{1/2}+\eps\right).
 \end{split}
\end{equation*}

On the other hand,
\begin{equation*}
 \begin{split}
&\eps  \left|\int_0^{+\infty}
  \iint_{\RR\times \RR} L_\eps(t,x)\cdot \nabla_v \Phi(t,x,v)f_\eps(t,x,v)\,dx\,dv\,dt\right|\\&
  \leq \eps \|\nabla_v\Phi\|_{L^\infty} \pi R^2\,\gamma\, \|f_\eps^0\|_{L^\infty}\int_0^R 
  \left( \int_{B(\xi_\eps(t),\|f_\eps^0\|_{L^\infty}^{-1/2})}\frac{dx}{|x-\xi_\eps(t)|}\right)\,dt\\
  &+\eps \|\nabla_v \Phi\|_{L^\infty}\,\gamma \|f_\eps^0\|_{L^\infty}^{1/2}\int_0^R
  \int_{\RR\setminus B(\xi_\eps(t),\|f_\eps^0\|_{L^\infty}^{-1/2})}\int_{\RR}f_\eps(t,x,v)\,dx\,dv\,dt\\
  &\leq C\|\nabla_v\Phi\|_{L^\infty}\left(\eps \|f_\eps^0\|_{L^\infty}^{1/2}+\eps \|f_\eps^0\|_{L^\infty}^{1/2}\right)\\
  &\leq C \eps \|f_\eps^0\|_{L^\infty}^{1/2}.
 \end{split}
\end{equation*}
Since $\eps \|f_\eps^0\|_{L^\infty}^{1/2}$ tends to zero as $\eps \to 0$ by assumption \eqref{hyp:ini-2-bis}, we infer that
$v^\perp \cdot \nabla_v f_\eps=\nabla_v\cdot(v^\perp f_\eps) \to 0$ in the sense of distributions.

Now, the sequence  $(f_{\eps_n})$ is uniformly bounded in $ L^\infty(\R_+,L^1(\RR\times \RR))$, thus it is relatively compact in 
$ L^\infty(\R_+,\mathcal{M}(\RR\times \RR))$ weak - $*$. Let $f$ be an accumulation point. In view of the previous estimates we obtain in the limit: 
$\nabla_v\cdot(v^\perp f)=0$ in the sense of distributions.

\end{proof}

\begin{proposition}
 \label{prop:radial-terms}Under the assumptions of Theorem \ref{thm:main}, there exists a subsequence $(f_{\eps_{n_k}})$  and there exists $f=f(t,x,|v|)\in 
 L^\infty(\R_+,\mathcal{M}(\RR\times \RR))$ such that $(f_{\eps_n})$ converges to $f$ 
in $L^\infty(\R_+, \mathcal{M}_+(\RR \times \RR))$ weak - 
$\ast$ and such that  $\rho=\int f\,dv$. Moreover, there exists a measure 
 $\nu_0\in L^\infty(\R_+,\mathcal{M}(\RR\times \mathbb{S}^1))$ 
 such that  for all $\Phi$ continuous on $\mathbb{S}^1$, 
$$\int_{\RR}\left(f_{\eps_{n_k}}(t,x,v)-f(t,x,|v|)\right)\Phi\left(\frac{v}{|v|}\right)|v|^2\,dv$$ converges to$$\int_{\mathbb{S}^1}\Phi(\theta)\: d\nu_0(\theta)$$ in the 
sense of distributions on $\R_+\times \RR$. In particular, we have:
 \begin{equation*}
\int_{\RR}v_1v_2f_{\eps_{n_k}}\,dv\to \int_{\mathbb{S}^1}\theta_1\theta_2d\nu_0(\theta)\quad\text{as } k\to +\infty\end{equation*}
 and \begin{equation*}
 \int_{\RR}(v_2^2-v_1^2)f_{\eps_{n_k}}\,dv  \to \int_{\mathbb{S}^1}(\theta_2^2-\theta_1^2)d\nu_0(\theta)\quad\text{as } k\to +\infty
 \end{equation*} in the sense of distributions on $\R_+\times \RR$.
\end{proposition}

\begin{proof} In view of Lemma \ref{lemma:radial}, which extends \cite[Lemma 3.3]{golse-sr}, we may argue exactly as in the 
 beginning of the proof of \cite[Theorem A]{golse-sr} (pages 802--803) to find a measure $\nu_0$ satisfying the previous properties. We do not provide the details here.

 \end{proof}

\subsection{Proofs of Theorems \ref{thm:main} and Theorem \ref{thm:main-structure} completed} We consider the subsequence $(\rho_{\eps_{n_k}})$ of $(\rho_{\eps_{n}})$, which we still denote
by $(\rho_{\eps_{n}})$ for simplicity. In order to prove Theorem \ref{thm:main} we have to pass to the limit in the weak formulation 
given by Proposition \ref{prop:weak-formulation-3}. Let $\Phi$ be a test function and let $t\geq 0$.
On the one hand, the compactness statements of Proposition \ref{prop:compactness} directly imply that
\begin{equation*}\begin{split}
 &\int_{\RR}\Phi(t,x)\rho_{\eps_n}(t,x)\,dx+ \Phi(t,\xi_{\eps_n}(t))-\int_{\RR}\Phi(0,x)\rho_{\eps_n}(0,x)\,dx-\Phi(0,\xi_{\eps_n}(0))\\
 &-\int_0^t \int_{\RR}\partial_t\Phi(s,x) \rho_{\eps_n}(s,x)\,ds\,dx-\int_0^t \partial_t \Phi(s,\xi_{\eps_n}(s))\,ds
\end{split}
\end{equation*}
converges to 
\begin{equation*}\begin{split}
 &\int_{\RR}\Phi(t,x)\rho(t,x)\,dx+\Phi(t,\xi(t))-\int_{\RR}\Phi(0,x)\rho(0,x)\,dx-\Phi(0,\xi(0))\\
 &-\int_0^t \int_{\RR}\partial_t\Phi(s,x) \rho(s,x)\,ds\,dx-\int_0^t \partial_t \Phi(s,\xi(s))\,ds
\end{split}
\end{equation*}
as $n$ tends to $\infty$.

We turn now to the nonlinear terms in Proposition \ref{prop:weak-formulation-3}. The sequence $(\rho_{\eps_n})$ is uniformly bounded in $L^\infty(\R_+,\mathcal{M}(\R^2))$ and
it satisfies the non-concentration property of Proposition \ref{prop:delort-2}. It was proved 
in \cite{Delort} (see also \cite{majda-93,Schochet}) that this, together with the  convergence of $(\rho_{\eps_n})$ to $\rho$, implies the convergence
of $\int \mathcal{H}_{\Phi(s,\cdot)}[\rho_{\eps_n}(s),\rho_{\eps_n}(s)]\,ds$ to $\int \mathcal{H}_{\Phi(s,\cdot)}[\rho(s),\rho(s)]\,ds$. By Proposition \ref{prop:delort-2}, this also yields the convergence of
$\int \mathcal{H}_{\Phi(s,\cdot)}[\rho_{\eps_n}(s),\delta_{\xi_{\eps_n}}(s)]\,ds$ to $\int \mathcal{H}_{\Phi(s,\cdot)}[\rho(s),\delta_{\xi}(s)]\,ds$
(this is done in the proof of (25) in \cite{miot-parme}).

We finally handle the last terms of Proposition \ref{prop:weak-formulation-3}. By virtue of Proposition \ref{prop:radial-terms} we already know that
\begin{equation*}\begin{split}
 \int_0^t \int_{\RR}&\left(D\nabla^\perp\Phi(s,x):\int_{\RR}v\otimes v f_{\eps_n}(s,x,v)\,dv\right)\,dx\,ds
\end{split}
\end{equation*}
converges to
\begin{equation*}
\int_0^t \int_{\RR}\left(D\nabla^\perp\Phi(s,x):\int_{\mathbb{S}^1}\theta\otimes \theta\,d\nu_0(s,x,\theta)\right)\,dx\,ds.
\end{equation*}
Moreover, since $(\eta_{\eps_n})$ is uniformly bounded, there exists $\alpha$ and $\beta$ in $L^\infty(\R_+,\R)$ 
such that, after extracting a subsequence (still denoted in the same way), $\eta_{\eps_n,1}\eta_{\eps_n,2}$ converges to $\alpha$ and $\eta_{\eps_n,2}^2-\eta_{\eps_n,1}^2$ 
to $2\beta$ in $L^\infty(\R_+)$ weak - $\ast$. But
$D\nabla^\perp \Phi(s,\xi_{\eps_n}(s))$ converges to  $D\nabla^\perp \Phi(s,\xi(s))$ locally uniformly on $\R_+$, so
$$\int_0^t \eta_{\eps_n}(s)\cdot [D\nabla^\perp\Phi(s,\xi_{\eps_n}(s))\: \eta_{\eps_n}(s)]\,ds$$ converges to
$$\int_0^t  (\partial_{11}-\partial_{22})\Phi(s,\xi(s))\alpha(s)+2\partial_{12} \Phi(s,\xi(s))\beta(s)\,ds.$$
Considering the measure
$$d\nu(t,x)=\int_{\mathbb{S}^1}\theta\otimes \theta d\nu_0(t,x,\theta)+\begin{pmatrix} -\beta(t) & \alpha(t) \\ \alpha(t) & {\beta(t)}\end{pmatrix}\delta_{\xi(t)}\in L^\infty(\R_+,\mathcal{M}(\RR)),$$
we conclude the proof.

\section{Proof of Theorem  \ref{thm:main-bis}}

We  begin with the derivation of the first equation for $\rho$.
Let $\eta:\R_+ \to [0,1]$ be smooth such that $\eta$ vanishes on $[0,1]$ and $\eta=1$ on $[2,+\infty)$ and set $\eta_\delta=\eta(\cdot/\delta)$, which converges to $1$ almost everywhere.
 
 Let $\Phi$ be a test function and
$$\Phi_\delta(t,x)=\Phi(t,x)\eta_\delta\left(|x-\xi(t)|\right),$$
so that $\Phi_\delta(t,\xi(t))=\partial_t \Phi_\delta(t,\xi(t))\equiv 0$ and $\nabla \Phi_\delta(t,\xi(t))\equiv 0$.

On the one hand,  by Lebesgue's dominated convergence theorem, $\int \Phi_\delta(t,x)\rho(t,x)\,dx$  tends to $\int \Phi(t,x)\rho(t,x)\,dx$ as $\delta\to 0$.

Next, as noted in Remark \ref{rem:bounded}, we have $E\in L^\infty_\textrm{\loc}(L^\infty)$.  Moreover, as $\rho\in L^1_{\text{loc}}(L^p)$ for $p>2$ the quantity $\frac{1}{|x-\xi|})\rho$ belongs to $L^1_{\text{loc}}$. According to Proposition \ref{prop:symm} we may reexpress
the nonlinear term of \eqref{NLE} as
\begin{equation*}
\begin{split}&\mathcal{H}_{\Phi_\delta(t)}
=\int_{\RR} \left(E^\perp(t,x)+ \gamma\frac{(x-\xi(t))^\perp}{|x-\xi(t)|^2} \right)\cdot \nabla \Phi_\delta(t,x)\rho(t,x)\,dx\\
&=\int_{\RR} \left(E^\perp(t,x)+ \gamma \frac{(x-\xi(t))^\perp}{|x-\xi(t)|^2} \right)\cdot\nabla \Phi(t,x)  \eta_\delta\left({|x-\xi(t)|}\right)\rho(t,x)\,dx\\
&+\int_{\RR} E^\perp(t,x) \cdot  
\frac{x-\xi(t)}{|x-\xi(t)|}\Phi(t,x)\eta'_\delta\left({|x-\xi(t)|}\right)\rho(t,x)\,dx\\
&=I_\delta+J_\delta,
\end{split}
\end{equation*}where we have used that $a^\perp\cdot a=0$.
On the one hand,  Lebesgue's dominated convergence theorem implies that $I_\delta$ converges to
$$ \int_{\RR} \left(E^\perp(t,x)+ \gamma\frac{(x-\xi(t))^\perp}{|x-\xi(t)|^2} \right)\cdot\nabla \Phi(t,x) \rho(t,x)\,dx$$ as $\delta \to 0$.
On the other hand, we have by H\"older's inequality
\begin{equation*}
 \begin{split}
  |J_\delta|&\leq \frac{C}{\delta}
  \int_{|x-\xi(t)|\leq 2 \delta} |E(t,x)||\rho(t,x)|\,dx
  \leq  \frac{C}{\delta}\|E(t)\|_{L^\infty}\|\rho(t)\|_{L^p}\delta^{2-\frac{2}{p}},\end{split}
\end{equation*}
so $J_\delta$ vanishes in the limit $\delta\to 0$.

Finally, we compute
\begin{equation*}
 \begin{split}
\partial_t \Phi_\delta (t,x)&=\partial_t \Phi(t,x)  \eta_\delta\left(|x-\xi(t)|\right)+ \Phi(t,x) \frac{1}{ \delta}\eta'\left(|x-\xi(t)|\right)\dot{\xi}(t)\cdot \frac{\xi(t)-x}{|x-\xi(t)|}
 \end{split}
\end{equation*}and using that $|\dot{\xi}(t)|\leq C$ we find as above that the integral $\int \partial_t \Phi_\delta(t,x)\rho(t,x)\,dx$ converges to $\int \partial_t \Phi(t,x)\rho(t,x)\,dx$ as 
$\delta \to 0$. Therefore, we have proved that $\rho$ satisfies the first equation of \eqref{syst:VW} in the sense of distributions. Inserting this equation in \eqref{NLE} for 
any function $\Phi$ not necessarily vanishing near $\xi(t)$, 
we infer that
$$\frac{d}{dt}\Phi(t,\xi(t))= \partial_t \Phi(t,\xi(t))+  E^\perp(t,\xi(t))\cdot \nabla \Phi(t, \xi(t)),$$
which yields the second equation for $\xi$.
\medskip

\noindent \textbf{Acknowledgments} {During the preparation of this worh the author has been 
partially supported by the French ANR projects SchEq ANR-12-JS-0005-01, GEODISP ANR-12-BS01-0015-01, and INFAMIE ANR-15-CE40-01.}


\begin{thebibliography}{99}






\bibitem{Arsenev}  A. A. Arsenʹev, \emph{Existence in the large of a weak solution of Vlasov's system of equations} (Russian), \u{Z}. Vy\u{c}isl. Mat. i Mat. Fiz. \textbf{15} (1975), 136--147, 276.

\bibitem{barre-chiron-masmoudi} J. Barr\'e, D. Chiron, T. Goudon and N. Masmoudi, \emph{From Vlasov-Poisson and Vlasov-Poisson-Fokker-Planck Systems to Incompressible Euler Equations: the case with finite charge}, preprint
arXiv:1502.07890, 2015.


\bibitem{bostan-finot-hauray} M. Bostan, A. Finot and M. Hauray, \emph{The effective Vlasov-Poisson system for strongly magnetized plasmas}, preprint arXiv:1511.00169, 2015.

\bibitem{brenier} Y. Brenier, \emph{Convergence of the Vlasov-Poisson system to the incompressible Euler
equations}, Comm. Partial Differential Equations \textbf{25} (2000), 737--754.


\bibitem{italiens-miot} S. Caprino, C. Marchioro, E. Miot and M. Pulvirenti, \emph{On the attractive plasma-charge model in 2-D},  Comm. Partial Differential Equations \textbf{37} (2012), no. 7, 1237--1272.

\bibitem{caprino-marchioro} S. Caprino and C. Marchioro, \emph{On the plasma-charge model}, Kin. Rel. Mod. \textbf{3}, no.2 (2010), 241--254.


\bibitem{bresiliens-miot} G. Crippa, M. C. Lopes Filho, E. Miot and H. J. Nussenzveig Lopes, 
\emph{Flows of vector fields with point singularities and the vortex-wave system}, Discrete and continuous dynamical systems \textbf{5} (2016), 2405--2417.



\bibitem{Delort} J.-M. Delort, \emph{Existence de nappes de tourbillon en dimension deux} (French) [Existence of vortex sheets in dimension two], J. Amer. Math. Soc. \textbf{4} (1991), no. 3, 553--586.



\bibitem{filbet-rodrigues} F. Filbet and L. M. Rodrigues,
\emph{Asymptotically stable particle-in-cell methods for the Vlasov-Poisson system with a strong external magnetic field}, 
SIAM J. Numer. Analysis (2016). 

\bibitem{frenod-sonnendrucker-98} E. Fr\'enod and E. Sonnendr\"ucker, \emph{Homogenization of the Vlasov equation and of the Vlasov-Poisson system with a strong external magnetic field}, Asymptot. Anal. \textbf{18} (1998), no. 3-4, 
193--213.

\bibitem{frenod-sonnendrucker-99}  E. Fr\'enod and E. Sonnendr\"ucker, \emph{Long time behavior of the two-dimensional Vlasov equation with a strong external magnetic field}, Math. Models Methods Appl. Sci. \textbf{10} (2000), no. 4, 539--553.

\bibitem{frenod-sonnendrucker-01}E. Fr\'enod and E. Sonnendr\"ucker, \emph{The Finite Larmor Radius Approximation},
SIAM J. Math. Anal. \textbf{32} (2001), no. 6, 1227--1247.



\bibitem{golse-sr} F. Golse and L. Saint-Raymond, \emph{The Vlasov-Poisson system with strong 
magnetic field}, J. Math. Pures Appl. (9) \textbf{78} (1999), no. 8, 791--817.

\bibitem{golse-sr-2} F. Golse and L. Saint-Raymond, \emph{The Vlasov-Poisson system with strong magnetic
field in quasineutral regime},  Mathematical Models and Methods in Applied
Sciences \textbf{13} (2003), no. 5, 661--714.


\bibitem{han-kwan} D. Han-Kwan, \emph{The three-dimensional Finite Larmor Radius Approximation}, Asymptot. Anal. \textbf{66} (2010), no.1, 9--33.

\bibitem{hauray-nouri}  M. Hauray and A. Nouri, \emph{Well-posedness of a diffusive gyro-kinetic model},
Ann. Inst. H. Poincar\'e Anal. Non Lin\'eaire \textbf{28} (2011), no. 4, 529--550.





\bibitem{lacave-miot} C. Lacave and E. Miot, \emph{Uniqueness for the vortex-wave system when the vorticity is constant near the point vortex},  SIAM J. Math. Anal. \textbf{41} 
(2009), no. 3, 1138--1163.



 \bibitem{LP} P. L. Lions and B. Perthame, \emph{Propagation of moments
and regularity  for the 3-dimensional Vlasov-Poisson system}, Invent.
Math. \textbf{105} (1991), 415--430.




\bibitem{majda-93} A. J. Majda, \emph{Remarks on weak solutions for vortex sheets with a distinguished sign},
Indiana Univ. Math. J.  \textbf{42} (1993), 921-939.

\bibitem{majda-bertozzi} A. J. Majda and A. L. Bertozzi, {\it Vorticity and incompressible flow}, Cambridge Texts in Applied Mathematics \text{27}. Cambridge University Press, Cambridge, 2002.


\bibitem{loeper} G. Loeper,
\emph{Uniqueness  of  the  solution  to  the  Vlasov-Poisson  system  with  bounded
density}, J. Math. Pures Appl. \textbf{86} (9)(2006), no. 1, 68--79.

\bibitem{livrejaune}C. Marchioro and M. Pulvirenti, {\it Mathematical
Theory of Incompressible Nonviscous Fluids}, Springer-Verlag, New York, 1994.

\bibitem{marchioro} C. Marchioro, \emph{On the Euler equations with a singular external velocity field},
Rend. Sem. Mat. Univ. Padova \textbf{89} (1990), 61--69.

\bibitem{marchioro-pulvirenti} C. Marchioro and M. Pulvirenti, \emph{On the vortex-wave system}

\bibitem{miot-16} E. Miot, \emph{On the gyrokinetic limit for the 2D Vlasov-Poisson system}, preprint, 2016.

\bibitem{miot-parme} E. Miot, \emph{Two existence results for the vortex-wave system},  Riv. Math. Univ. Parma \textbf{3} (2012), no. 1, 131--146.



\bibitem{ukai}  S. Okabe and T. Ukai,\emph{On classical solutions in the large in time of the two-dimensional Vlasov equation}, Osaka J. Math. \textbf{15} (1978), 245--261.


\bibitem{poupaud} F. Poupaud, \emph{Diagonal defect measures, adhesion dynamics and Euler equation}, Methods Appl. Anal. \textbf{9} (2002), no. 4, 533--561.
%

\bibitem{SR-01}L. Saint-Raymond, \emph{The gyrokinetic approximation for the Vlasov-Poisson system}, Math. Mod. Meth. Appl. Sci. \textbf{10} (9) (2000), 1305--1332.


\bibitem{SR-02} L. Saint-Raymond,  \emph{Control of large velocities in the two-dimensional gyrokinetic approximation},
J. Math. Pures Appl. (9) \textbf{81} (2002), no. 4, 379--399.



\bibitem{Schochet} S. Schochet, 
\emph{The weak vorticity formulation of the 2-D Euler equations and concentration-cancellation},
Comm. Partial Differential Equations \textbf{20} (1995), no. 5--6, 1077--1104.



\end{thebibliography}
\end{document}